\documentclass[12pt,a4paper]{amsart}         
\usepackage[margin=3cm]{geometry}
\usepackage[backend=bibtex,url=false,doi=true,isbn=false,giveninits=true,style=numeric,sorting=nyt]{biblatex}
\usepackage{amsmath,amsfonts,amssymb,amsthm}
\usepackage{mathrsfs,epsfig,epstopdf,url}
\usepackage{enumerate}
\usepackage[table]{xcolor}
\usepackage{multirow}

\theoremstyle{definition}
\newtheorem{definition}{Definition}%

\theoremstyle{plain}
\newtheorem{theorem}[definition]{Theorem}
\newtheorem{lemma}[definition]{Lemma}
\newtheorem{proposition}[definition]{Proposition}
\newtheorem{problem}{Problem}
\newtheorem{corollary}[definition]{Corollary}
\newtheorem{remark}[definition]{Remark}

\newcommand{\E}{\mathsf{E}}
\newcommand{\D}{\mathsf{D}}

\newcommand{\K}{\mathcal{K}}

\title{Simplicity conditions for binary orthogonal arrays}

\author{Claude Carlet}
\address{Universities of Paris 8 and Paris 13, CNRS LAGA (UMR 7539)\\
	Dept of Math. Univ. Paris 8\\
	2 rue de la Liberté\\
	F-93 526 Saint-Denis Cedex, France}
\address{Department of Informatics\\
	University of Bergen\\
	PB 7803, N-5020 Bergen, Norway}
\email{claude.carlet@gmail.com}
\author{Rebeka Kiss}
\address{Bolyai Institute \\
	University of Szeged \\
	Aradi v\'ertan\'uk tere 1\\
	H-6720 Szeged, Hungary}
\email{Kiss.Rebeka@stud.u-szeged.hu}
\author{G\'abor P. Nagy}
\address{Department of Algebra \\
	Budapest University of Technology and Economics\\
	M\H{u}egyetem rkp 3\\
	H-1111 Budapest, Hungary}
\address{Bolyai Institute \\
	University of Szeged \\
	Aradi v\'ertan\'uk tere 1\\
	H-6720 Szeged, Hungary}
\email{nagy.gabor.peter@ttk.bme.hu}

\thanks{The research of C. Carlet is partly supported by the Trond Mohn Foundation and Norwegian Research Council. For R. Kiss and G.P. Nagy, support has been provided from the National Research, Development and Innovation Fund of Hungary, financed under the 2018-1.2.1-NKP funding scheme, within the SETIT Project 2018-1.2.1-NKP-2018-00004. The research of G.P. Nagy is partly supported by NKFIH-OTKA Grant SNN 132625. }

\keywords{Orthogonal array; Correlation-immune Boolean function; Rao's Bound; Linear Programming Bound}
\subjclass[2010]{05B05}

\begin{document}

\begin{abstract}
It is known that correlation-immune (CI) Boolean functions used in the framework of side channel attacks need to have low Hamming weights. The supports of CI functions are (equivalently) simple orthogonal arrays, when their elements are written as rows of an array. The minimum Hamming weight of a CI function is then the same as the minimum number of rows in a simple orthogonal array. In this paper, we use Rao's Bound to give a sufficient condition on the number of rows, for a binary orthogonal array (OA) to be simple. We apply this result for determining the minimum number of rows in all simple binary orthogonal arrays of strengths 2 and 3;  we show that this minimum is the same in such case as for all OA, and we extend this observation to some OA of strengths $4$ and $5$. This allows us to reply positively, in the case of strengths 2 and 3, to a question raised by the first author and X. Chen on the monotonicity of the minimum Hamming weight of 2-CI Boolean functions, and to partially reply positively to the same question in the case of strengths 4 and 5.
\end{abstract}

\maketitle

\section{Introduction} \label{sec:intro}

In cryptography, {\em correlation immune} (CI) functions are those Boolean functions over $\mathbb{F}_2^k$ whose output distribution does not change when at most $t$ input bits are fixed, where $t\leq k$ is the correlation immunity order, whatever is the choice of these input bits and whatever are the values to which they are fixed. As shown in \cite{xiaomassey1988}, they are those $k$-variable Boolean functions whose Fourier transform $\widehat{f}(a)=\sum_{x\in \mathbb{F}_2^k}f(x)(-1)^{a\cdot x}$ (where ``$\cdot$" is the usual inner product in $\mathbb{F}_2^k$) vanishes for all nonzero inputs $a\in \mathbb{F}_2^k$ of Hamming weight at most $t$. In other words, the supports of these functions are unrestricted  (i.e. linear or nonlinear) binary codes of dual distance at least $t+1$. The correlation immunity of a function $f$ allows the resistance against the {\em Siegenthaler correlation attack} on the stream ciphers using $f$ as a combining function (see \cite{carlet2021boolean} for more details). CI functions can also be used for implementing the {\em rotating S-box masking} counter-measure against side channel attacks (see \cite{carlet2021boolean} as well). We can reduce the cost of this counter-measure by finding, for given $k$ and $t$, the minimum Hamming weight $w_{k,t}$ of $t$-th order CI-functions in $k$ variables, that is the minimal size of their supports, and then by using a CI function of such weight in the implementation. 
 The first author and Guilley \cite{MR3784600,MR3287682} published a table containing the values of $w_{k,t}$ for small $k,t$. It is difficult to give these values even for small parameters, this is demonstrated by the facts that the table is limited to $k\leq 13$ and even then, there are missing values in the table. 

CI-functions are closely related to orthogonal arrays, introduced by C.R. Rao \cite{Rao1947} in 1947. Let $N,t,k$ be positive integers, $t\leq k$, and $S$ a finite set of cardinality $s$. An $N\times k$ array $A$ with entries from $S$ is said to be an \textit{orthogonal array with $s$ symbols, strength $t$, and index $\lambda$,} if every $N\times t$ subarray of $A$ contains each $t$-tuple based on $S$ exactly $\lambda$ times as a row. We will denote such an array by $\mathit{OA}(N,k,s,t)$. We have $\lambda = N/s^t$. An orthogonal array is called \textit{simple} if the rows are distinct. Supports of $t$-th order CI-functions give simple binary orthogonal arrays with strength $t$, if their elements are written as rows, and vice versa. 

In the theory of orthogonal arrays, for both simple and general orthogonal arrays, the main question is to give -- for given numbers $k$ of columns and $s$ of symbols, and for strength $t$ -- the minimum value of $N$ for which an orthogonal array $\mathit{OA}(N,k,s,t)$ exists with $N$ rows. We will denote this value by $F^*(k,s,t)$ for simple orthogonal arrays (we have then $w_{k,t}=F^*(k,2,t)$) and by $F(k,s,t)$ for general orthogonal arrays. This problem is very hard even for the smallest parameters $s=t=2$. In fact, a binary orthogonal array of strength $2$ with $k$ columns and $k+1$ rows is equivalent to a Hadamard matrix of order $k+1$. A \textit{Hadamard matrix} of order $n$ is an $n\times n$ matrix whose entries are either $+1$ or $-1$, and whose rows are mutually orthogonal. The famous Hadamard conjecture proposes that a Hadamard matrix of order $n$ exists if and only if $n$ is divisible by $4$. Equivalently in our notation: $F(k,2,2)=k+1$ if and only if $k$ is congruent to $3$ modulo $4$.

For some lower bounds on the number $N$ of rows, it is known that if an $\mathit{OA}(N,k,s,t)$  attains this special bound, then it is simple. For example, this is true for
the Friedman-Bierbrauer bound \cite{Bierbrauer1995}  
\[N\geq s^k\left(1-\frac{(s-1)k}{s(t+1)}\right).\]
Indeed, it is seen from the proof that any multiplicity greater than $1$ makes the inequality strict. For binary orthogonal arrays of strength $t\geq (2k-2)/3$, the bound $N\geq 2^{k-1}$ implies simplicity in the case of equality, see \cite{Khalyavin2010}. 

In \cite{MR3287682}, the first author and Guilley asked the the following question:
\begin{problem}[Carlet-Guilley] \label{prob:CG}
Is $F^*(k,2,t)$ a monotone non-decreasing function when $k$ grows and $t$ remains fixed?
\end{problem}
The same question for $F(k,s,t)$ is trivial, since an $\mathit{OA}(N,k,s,t)$ gives rise to an $\mathit{OA}(N,k-1,s,t)$ by deleting one of the columns. Moreover, if $F(k,s,t)=F^*(k,s,t)$, then
\[F^*(k,s,t)\leq F(k+1,s,t) \leq F^*(k+1,s,t).\]
Hence, the solution of the following problem would imply an answer to the problem posed by the first author and Guilley:
\begin{problem} \label{prob:new}
Find all parameters $k,s,t$ such that $F(k,s,t)=F^*(k,s,t)$.
\end{problem}
In this paper, we give a partial answer to Problem \ref{prob:new}. Our main theoretical result is the following:
\begin{theorem} \label{thm:main}
Let $A$ be an $\mathit{OA}(N,k,s,2u)$. Define the integer
\[M(k,s,2u)=\sum_{j=0}^u \binom{k}{j}(s-1)^j.\]
\begin{enumerate}[(i)]
\item If $A$ has a row of multiplicity $\rho$, then $N\geq \rho\, M(k,s,2u)$.
\item If $N<2\,M(k,s,2u)$, then $A$ is simple. If $N<3\,M(k,s,2u)$, then each row of $A$ has multiplicity at most $2$.
\item If $k\geq 5$, $s=2$, $u=2$ and 
\[N=2\,M(k,2,4)=k^2+k+2, \]
then either $A$ is simple, or $k=5$ and $A$ is obtained by the juxtaposition of two identical arrays $\mathit{OA}(16,5,2,4)$. 
\end{enumerate} 
\end{theorem}
Part (ii) of Theorem \ref{thm:main} implies a sufficient condition for the parameters $k,s,t$ to fulfill Problem \ref{prob:new}:
\begin{corollary} \label{cor:1}
If $t$ is even and
\[F(k,s,t) < 2\, M(k,s,t)\]
then
\[F^*(k,s,t)=F(k,s,t). \qed\]
\end{corollary}
Notice that the integer $M(k,s,t)$ is the lower bound for the number of rows in an orthogonal array with $k$ columns, $s$ symbols and strength $t$, given in Rao's famous theorem \cite[Theorem 2.1]{MR1693498}:
\begin{align}
F(k,s,t)\geq M(k,s,t) \qquad \text{for all positive integers $k,s,t$.}
\end{align} 
For part (iii) of Theorem \ref{thm:main}, we observe that $M(5,2,4)=16$, and up to equivalence, there is a unique $\mathit{OA}(16,5,2,4)$. If we assume that such an array has an all-$0$ row, then all its rows have an even number of $1$s.

\begin{table}[]\footnotesize
	\caption{Number of rows in minimal simple orthogonal arrays with given number of columns and given strength}
	\label{tab:minrows}
	\begin{tabular}{c||c|cc|cc|cc|cc|cc|cc|}
		$k\backslash t$   & \textbf{1} & \textbf{2} & \textbf{3} & \textbf{4} & \textbf{5}  & \textbf{6} & \textbf{7}  & \textbf{8} & \textbf{9} & \textbf{10} & \textbf{11} & \textbf{12} & \textbf{13} \\ \hline\hline
		\textbf{1}  & \cellcolor{black!15}2  &&&&&&&&&&&&   \\ \hline
		\textbf{2}  & \cellcolor{black!15}2  & \cellcolor{black!15}4  &&&&&&&&&&&   \\ \hline
		\textbf{3}  & \cellcolor{black!15}2  & \cellcolor{black!40}4  & \cellcolor{black!15}8  &&&&&&&&&&   \\ \hline
		\textbf{4}  & \cellcolor{black!15}2  & \cellcolor{black!40}8  & \cellcolor{black!40}8  & \cellcolor{black!15}16  &&&&&&&&&   \\ \hline
		\textbf{5}  & \cellcolor{black!15}2  & \cellcolor{yellow!20}8  & \cellcolor{black!40}16  & \cellcolor{black!40}16  & \cellcolor{black!15}32  &&&&&&&&   \\ \hline
		\textbf{6}  & \cellcolor{black!15}2  & \cellcolor{yellow!20}8  & \cellcolor{yellow!20}16  & \cellcolor{black!40}32  & \cellcolor{black!40}32  & \cellcolor{black!15}64  &&&&&&&   \\ \hline
		\textbf{7}  & \cellcolor{black!15}2  & \cellcolor{yellow!20}8  & \cellcolor{yellow!20}16  & \cellcolor{black!40}64  & \cellcolor{black!40}64  & \cellcolor{black!40}64  & \cellcolor{black!15}128  &&&&&&   \\ \hline
		\textbf{8}  & \cellcolor{black!15}2  & \cellcolor{yellow!20}12  & \cellcolor{yellow!20}16  & \cellcolor{green!20}64  & \cellcolor{black!40}128  & \cellcolor{black!40}128  & \cellcolor{black!40}128  & \cellcolor{black!15}256  &&&&&   \\ \hline
		\textbf{9}  & \cellcolor{black!15}2  & \cellcolor{yellow!20}12  & \cellcolor{yellow!20}24  & \cellcolor{magenta!30}128  & \cellcolor{green!20}128  & \cellcolor{black!40}256  & \cellcolor{black!40}256  & \cellcolor{black!40}256  & \cellcolor{black!15}512  &&&&   \\ \hline
		\textbf{10} & \cellcolor{black!15}2  & \cellcolor{yellow!20}12  & \cellcolor{yellow!20}24  & \cellcolor{magenta!30}128  & \cellcolor{magenta!30}256  & \cellcolor{black!40}512  & \cellcolor{black!40}512  & \cellcolor{black!40}512  & \cellcolor{black!40}512  & \cellcolor{black!15}$1\,024$  &&&   \\ \hline
		\textbf{11} & \cellcolor{black!15}2  & \cellcolor{yellow!20}12  & \cellcolor{yellow!20}24  & \textbf{A} & \textbf{A'} & \cellcolor{green!20}512  & \cellcolor{black!40}$1\,024$  & \cellcolor{black!40}$1\,024$  & \cellcolor{black!40}$1\,024$  & \cellcolor{black!40}$1\,024$  & \cellcolor{black!15}$2\,048$  &&   \\ \hline
		\textbf{12} & \cellcolor{black!15}2  & \cellcolor{yellow!20}16  & \cellcolor{yellow!20}24  & \textbf{A} & \textbf{A'} & \textbf{B} & \cellcolor{green!20}$1\,024$  & \cellcolor{black!40}$2\,048$  & \cellcolor{black!40}$2\,048$  & \cellcolor{black!40}$2\,048$  & \cellcolor{black!40}$2\,048$  & \cellcolor{black!15}$4\,096$  &   \\ \hline
		\textbf{13} & \cellcolor{black!15}2  & \cellcolor{yellow!20}16  & \cellcolor{yellow!20}32  & \textbf{A} & \textbf{A'} & \textbf{C} & \textbf{B'} & \cellcolor{black!40}$4\,096$  & \cellcolor{black!40}$4\,096$  & \cellcolor{black!40}$4\,096$  & \cellcolor{black!40}$4\,096$  & \cellcolor{black!40}$4\,096$  & \cellcolor{black!15}$8\,192$        \\ \hline
	\end{tabular}
\end{table}
We conclude this section with Table \ref{tab:minrows}, which shows the values of $F^*(k,2,t)$ for $1\leq k,t \leq 13$; it is a reproduction of the tables in \cite{MR3784600,MR3287682,MR4001794}. Using old and new computational results, and Theorem \ref{thm:main}, we were able to fill in new entries in Table \ref{tab:minrows}, denoted by capital letters. For previously known entries we colored the cells; the meaning of the colors are explained below.
\begin{description}
\item[gray] The light gray fields are trivial. The dark gray fields are consequences of the Fon-Der-Flaass Theorem \cite{MR2465419}. 
\item[yellow] The yellow fields are related to the constructions of Hadamard matrices, to the famous Hadamard Conjecture, and to a recent conjecture by the first author and Chen, see section \ref{sec:strength2_4} for details. 
\item[green] The values equal to Delsarte's LP Bound, and the construction is given by a linear code of codimension $2$, see \cite{MR3784600,MR3287682,MR4001794}. 
\item[red] The first author and Guilley \cite{MR3287682} contributed the values by using the Satisfiability Modulo Theory (SMT) tool \texttt{z3} \cite{z32008}. The upper bound follows from a well-known construction that is related to shortening of the non-linear binary Kerdock code of length $16$, see \cite{KERDOCK1972182}. 
\item[A, A'] $A=128$ and $A'=256$, see Proposition \ref{pr:table}(A). 
\item[B, B'] $B=768$ and $B'=1\,536$. The values equal to Delsarte's LP Bound. The existence and uniqueness of an $\mathit{OA}(1\,536,13,2,7)$ has been shown recently by Krotov \cite{Krotov2020}. See Proposition \ref{pr:table}(B) for an independent construction. 
\item[C] $C=1\,024$, see Proposition \ref{pr:table}(C) and \cite{MR4001794}. 
\end{description}

\section{Preliminary results}

In this section, we collected some preliminary results and notation on the minimum number of rows of an orthogonal array with $k$ rows, $s$ symbols and strength $t$. Recall the definition
\begin{align*}
F(k,s,t) &= \min \{ N \mid \exists \,\mathit{OA}(N,k,s,t)\}, \\
F^*(k,s,t) &= \min \{ N \mid \exists \text{ simple } \mathit{OA}(N,k,s,t)\}. \\
\end{align*}

\begin{lemma}
\begin{align}
F(k,s,t) &\leq F^*(k,s,t), \label{eq:FeqF*}\\
F(k,s,t) &\leq F(k+1,s,t), \label{eq:Fmonotone} \\
2\,F(k,2,2u) &= F(k+1,2,2u+1), \label{eq:2F} \\
2\,F^*(k,2,2u) &= F^*(k+1,2,2u+1). \label{eq:double}
\end{align}
\end{lemma}
\begin{proof}
\eqref{eq:FeqF*} and \eqref{eq:Fmonotone} are trivial. \cite[Theorem 2.24]{MR1693498} and \cite[Corollary 2.25]{MR1693498} imply \eqref{eq:2F}. \eqref{eq:double} holds by \cite[Proposition 2.6]{MR3784600}.
\end{proof}

\begin{remark}
Equation \eqref{eq:double} implies that it suffices to deal with orthogonal arrays of even strength $t=2u$ when studying the Carlet-Guilley problem and Problem \ref{prob:new}. This also shows that in the case of binary orthohonal arrays ($s=2$), one can use Theorem \ref{thm:main} to investigate the simplicity of arrays of odd strength. 
\end{remark}
\begin{remark}
For all integer $m$, duals of certain double-error-correcting BCH codes provide arrays $\mathit{OA}(2^{2m+1},2^m+1,2,5)$, and $\mathit{OA}(2^{2m},2^m,2,4)$ by \eqref{eq:double}. (See \cite[page 103]{MR1693498}.) If $k$ is an integer with $2^{m-1}<k\leq 2^m$, then
\[F(k,2,4)\leq F(2^m,2,4)\leq 2^{2m}<4k^2.\]
By Rao's Bound, $F(k,2,4)\geq (k^2+k+2)/2$. This shows that asymptotically, $F(k,2,4)$ and $F(k,2,5)$ are quadratic functions of $k$. 
\end{remark}

For tuples $u,v \in \{0,\ldots,s-1\}^k$, $w_H(u)$ denotes the Hamming weight, and 
\[uv^T=\sum_{i=1}^k u_iv_i\]
denotes the usual inner product (sometimes also denoted by $u\cdot v$ or by $\langle u,v\rangle$). For a matrix $H$ with complex entries, $H^*$ is the conjugate transpose of $H$. In particular, for complex (row) vectors $u,v \in \mathbb{C}^n$, 
\[uv^*=\sum_{i=1}^n u_i\bar{v}_i.\]
The $2$-norm of $u\in \mathbb{C}^n$ is
\[\|u\|=\sqrt{uu^*}.\]

Fix a primitive $s$-th root of unity $\zeta$. Let $A$ denote an $N\times k$ array with entries from $\{0,\ldots,s-1\}$. The $i$-th row of $A$ is denoted by $a_i$. For $1\leq i \leq N$ and $v\in \{0,\ldots,s-1\}^k$, we write:
\begin{equation}
\alpha_{i,v}=\zeta^{a_iv^T}.
\end{equation}
Clearly, for the zero vector $v=0$, we have $\alpha_{i,0}=1$. For any $v,v'$, we have
\[\alpha_{i,v}\alpha_{i,v'}=(\zeta^{a_iv^T}) (\zeta^{a_i(v')^T}) =\zeta^{a_i(v+v')^T}=\alpha_{i,v+v'},\]
and
\[\bar\alpha_{i,v}=\zeta^{-a_iv^T}=\zeta^{a_i(-v)^T}=\alpha_{i,-v}.\]

\begin{lemma} \label{lm:alphas}
The following statements are equivalent:
\begin{enumerate}[(i)]
\item The array $A$ is an $\mathit{OA}(N,k,s,t)$.
\item $\sum_{i=1}^N \alpha_{i,v}=0$ for any $v\in \{0,\ldots,s-1\}^k$ with $1\leq w_H(v)\leq t$. 
\item $\sum_{i=1}^N \alpha_{i,v}\bar{\alpha}_{i,v'}=0$ for any $v,v'\in \{0,\ldots,s-1\}^k$ with $w_H(v)+w_H(v')\leq t$. 
\end{enumerate}
\end{lemma}
\begin{proof}
The equivalence of (i) and (ii) is precisely \cite[Theorem 3.30]{MR1693498}. Setting $v'=0$, we obtain (ii) from (iii). For any $v,v'$, we have $\alpha_{i,v}\bar{\alpha}_{i,v'}=\alpha_{i,v-v'}$. As $w_H(v-v')\leq w_H(v)+w_H(v')\leq t$, (ii) implies (iii). 
\end{proof}

\begin{remark}
For binary arrays ($s=2$), Lemma \ref{lm:alphas}(ii) is the Xiao-Massey characterization of $k$-variable $t$-CI Boolean functions, see \cite{xiaomassey1988} or \cite[Theorem 2.2]{MR3784600}.
\end{remark}

\section{The proof of the main theorem} \label{sec:main}

The proof of \cite[Theorem 2.1]{MR1693498} is based on the introduction of two matrices $H$ and $Q$. We shall see that the same matrices can be used for proving our result. 
\begin{proof}[Proof of Theorem \ref{thm:main}]
Without loss of generality, we assume that the entries of $A$ are from $\{0,\ldots,s-1\}$. For any $0\leq j \leq u$, we define the $N\times \binom{k}{j}(s-1)^j$ matrix $H_j$ in the following way. The columns of $H_j$ are indexed with the tuples $v\in \{0,\ldots,s-1\}^k$ of Hamming weight $j$. For $1\leq i \leq N$ and tuple $v$ with $w_H(v)=j$, the entry $(i,v)$ of $H_j$ is $\alpha_{i,v}$. 
\\
The matrix:
\[H=[H_0 \, H_1 \, \cdots \, H_u]\]
has $N$ rows and 
\[M=\sum_{j=0}^u \binom{k}{j}(s-1)^j=M(k,s,2u)\]
columns. Any two columns of $H$ are orthogonal complex vectors by Lemma \ref{lm:alphas}(iii). Moreover, if column $h$ of $H$ is indexed by the tuple $v$, then
\[h^*h=\sum_{i=1}^N \bar{\alpha}_{i,v}{\alpha}_{i,v}=N.\]
This means that $H^* H = N I$, and the columns of $\frac{1}{\sqrt{N}} H$ form an orthonormal set of vectors in $\mathbb{C}^N$. This set can be extended into an orthonormal basis of $\mathbb{C}^N$. In other words, one can add columns to $\frac{1}{\sqrt{N}} H$ such that one obtains an $N\times N$ unitary matrix $Q$. Each row of $Q$ has the form $[u \, u']$, where $u$ is a vector of length $M$, with entries $\frac{\zeta^{a_iv^T}}{\sqrt{N}}$. In particular, 
\begin{align} \label{eq:ulength}
\|u\|=\sqrt{M/N}, \quad \|u'\|=\sqrt{1-M/N}. 
\end{align}
Let us assume that the rows $i_1,\ldots,i_\rho$ of $A$ are equal. Then, the rows $i_1,\ldots,i_\rho$ of $H$ are equal, and, the rows $i_1,\ldots,i_\rho$ of $Q$ have the form
\[[u \, u_{(r)}], \qquad r=1,\ldots,\rho.\]
The rows of $Q$ form an orthonormal basis, thus for all $1\leq r\neq s\leq \rho$,
\begin{align} \label{eq:uort}
0=uu^*+u_{(r)}u_{(s)}^*. 
\end{align}
Assume that $N<\rho M$. Then  \eqref{eq:ulength} and \eqref{eq:uort} imply
\begin{align} \label{eq:spherical}
u_{(r)}u_{(s)}^* < -\frac{1}{\rho} \qquad (r\neq s).
\end{align}
We have
\begin{align*}
0\ & \leq \left\|\sum_{r=1}^\rho u_{(r)}\right\|^2 \\
&=\sum_{s,r=1}^\rho u_{(r)}u_{(s)}^* \\
&=\rho \left(1-\frac{M}{N}\right) + \sum_{r\neq s} u_{(r)}u_{(s)}^* \\
&<1-\frac{\rho M}{N},
\end{align*}
using \eqref{eq:spherical} in the last step. The assumption $N<\rho M$ makes the right hand side negative, a contradiction. This proves (i). Part (ii) is a straightforward consequence of (i).

For the rest of the proof, $A$ denotes a non-simple $\mathit{OA}(k^2+k+2,k,2,4)$ with $k\geq 5$. By reordering the rows of $A$, and adding a fixed row to all rows modulo $2$, we may assume that the first two rows of $A$ are all $0$s. We use the notation $H_i$, $i=0,1,2$, $H$ and $Q$ from above. Recall that $H$ has $N$ rows and $N/2$ columns. As $\zeta=-1$, the entries of $H$ are $\pm 1$. The key observation is the following:
\begin{itemize}
\item[(*)] \textit{In rows $3,\ldots,N$, the number of $1$s is either $\ell_1$ or $\ell_2$, where}
\[\ell_{1,2}=\frac{k+1\pm\sqrt{k-1}}{2}.\]
\end{itemize}
Let us prove this. As the first two rows of $A$ are all-zeros, the first two rows of $Q$ have the form $[u \, u']$ and $[u \, u'']$, where
\[u=\left[\frac{1}{\sqrt{N}}\,\cdots\,\frac{1}{\sqrt{N}}\right].\]
Using the fact that $N=2 M(k,2,4)$, we show $u''=-u'$ in the same way as above. Let $[v\,v']$ be row $i$ of $Q$ with $i\geq 3$. This is orthogonal to the first two rows, hence,
\begin{align*}
0&=uv^T+u'(v')^T, \\
0&=uv^T+u''(v')^T=uv^T-u'(v')^T.
\end{align*}
This implies $uv^T=0$. This means that among the entries of $v$, $\frac{1}{\sqrt{N}}$ and $-\frac{1}{\sqrt{N}}$ occur equally often. In terms of $H$, this means that in this row, $1$ occurs $N/4$ times. 
\\
Let $\ell$ denote the number of $1$s in row $i$ of $A$. $H_0$ has one column, which consists of all $1$s. In row $i$ of $H_1$, the number of $1$s is $k-\ell$. In row $i$ of $H_2$, the number of $1$s is
\[\binom{\ell}{2}+\binom{k-\ell}{2}.\]
Hence, for the number of $1$s in row $i$ of $H$, we have
\[1+k-\ell+\binom{\ell}{2}+\binom{k-\ell}{2}=\frac{k^2+k+2}{4}.\]
Hence, we have $\ell^2-(k+1)\ell+(k^2+k+2)/4=0$, which implies (*). 

Immediate consequences are that $\kappa=\sqrt{k-1}$ is an integer, $N=k^2+k+2$ can be written as $N=\kappa^4+3\kappa^2+4$, and $\ell_{1,2}=(\kappa^2\pm\kappa+2)/2$.

Let us construct the array $A'$ by selecting all rows of $A$ that start with three zeros. We get
\[A'=\begin{bmatrix}
0&0&0&0&\cdots&0\\
0&0&0&0&\cdots&0\\
0&0&0\\
0&0&0&&B\\
0&0&0
\end{bmatrix}\]
where $B$ is a subarray with $N/8-2$ rows and $k-3$ columns. Since $A$ has strength $4$, then according to Lemma \ref{lm:alphas}, columns $4$ to $k$ of $A'$ have a number of $1$s equal to their number of $0$s, that equals then $N/16$. Let $a$ denote the number of rows of weight $\ell_1$ in $B$. The total number of $1$s in $B$ is
\begin{equation}
a\ell_1+(N/8-2-a)\ell_2=N/16\cdot (k-3).
\end{equation}
We reorder to get:
\begin{equation}
a(\ell_1-\ell_2)=N(k-3)/16-(N-16)\ell_2/8.
\end{equation}
Now, $\ell_1-\ell_2=\kappa=\sqrt{k-1}$. Also, the right hand side can be expanded into a polynomial of $\kappa$. This yields:
\begin{eqnarray*}
16\, a\, \kappa &=&(\kappa^4+3\kappa^2+4)(\kappa^2-2)- (\kappa^4+3\kappa^2-12)(\kappa^2-\kappa+2)\\&=& \kappa^5-4\kappa^4+3\kappa^3+4\kappa^2-12\kappa+16.
\end{eqnarray*}
We obtain that $16\equiv 0\pmod{\kappa}$, that is $\kappa$ divides $16$, and since by assumption, we have $k\geq 5$, that is, $\kappa\geq 2$, then we have $\kappa\in\{2,4,8,16\}$. \\If $\kappa\in\{4,8,16\}$, then $-12\kappa+16\equiv 0\pmod{64}$, that is, $3\kappa \equiv4\pmod{16}$. This implies $\kappa \equiv 12 \pmod{16}$ (since the inverse of 3 modulo 16 equals 11), a contradiction.
\\
Let us then consider the case $\kappa=2$. Then $k=5$, $N=32$, $\ell_1=4$ and $\ell_2=2$. Since $A$ has $30$ non-zero rows, and $\binom{5}{2}+\binom{5}{4}=15$, each nonzero row has multiplicity $2$. In other words, $A$ is twice an $\mathit{OA}(16,5,2,4)$. This finishes the proof of (iii).  
\end{proof}

\section{Simple arrays of strength 2 and 4} \label{sec:strength2_4}

In the special case of orthogonal arrays of strength $2$, we solve Problem \ref{prob:new}, and this allows us to give an affirmative answer to Problem \ref{prob:CG}. 

\begin{proposition} \label{pr:F22}
For $k\geq 2$, we have $F^*(k,2,2)=F(k,2,2)$. In particular, the sequence $F^*(k,2,2)$ is non-decreasing.
\end{proposition}
\begin{proof}
For any positive integer $h$, a classical Hadamard matrix $H_{2^h}$ is the matrix of the Hadamard Fourier transform, equal to the Kronecker product $H_2\otimes \cdots \otimes H_2$ of the matrix:
\[H_2=\begin{bmatrix} 1&1\\ 1&-1\end{bmatrix}\]with itself.
This implies 
\begin{align} \label{eq:Hadaconstr}
F(2^h-1,2,2)=2^h.
\end{align}
Given a positive integer $k$, let $h$ be the positive integer such that $2^{h-1}\leq k \leq 2^h-1$, then %
 \eqref{eq:Fmonotone} and \eqref{eq:Hadaconstr} imply:
\[F(k,2,2) \leq F(2^h-1,2,2) = 2^h \leq 2k. \]
As $M(k,2,2)=k+1$, we can apply Corollary \ref{cor:1} to obtain $F(k,2,2)=F^*(k,2,2)$. 
\end{proof}

The solution of Problem \ref{prob:new} has an implication to a recent conjecture by the first author and Chen. In \cite[Conjecture 2.8]{MR3784600}, the authors asked if 
\begin{align}\label{eq:conjCC} \tag{CC}
F^*(k,2,3)=8\left\lceil\frac{k}{4}\right\rceil.
\end{align}
Wang proved in \cite[Theorem 3.7]{MR4001794} that \eqref{eq:conjCC} and the Hadamard conjecture are equivalent. However, Wang's proof is incomplete, since no explanation is given for $w_{4\lambda+\varepsilon,3}\leq w_{4\lambda+4,3}$, where $1\leq \varepsilon \leq 3$ (and $w_{n,t}=F^*(n,2,t)$). In fact, this follows from Proposition \ref{pr:F22}. In order to be self-contained, we present a complete proof of the two conjectures. 
\begin{proposition}
The Hadamard conjecture is equivalent with Conjecture \ref{eq:conjCC}.
\end{proposition}
\begin{proof}
Recall that \cite[Theorem 7.5]{MR1693498} states that orthogonal arrays $\mathit{OA}(4\lambda , 4\lambda -1, 2, 2)$ and/or $\mathit{OA}(8\lambda , 4\lambda , 2, 3)$ exist (and then $F(4\lambda-1,2,2) \leq 4\lambda$) if and only if there exists a Hadamard matrix of order $4\lambda $. According to Rao's Bound, we have $F(4\lambda-1,2,2) \geq 4\lambda$. The Hadamard conjecture is then equivalent with:
\begin{align} \label{eq:conjHada} \tag{Ha}
F(4\lambda-1,2,2)= 4\lambda
\end{align}
for all positive integer $\lambda$. By \eqref{eq:double} and Proposition \ref{pr:F22}, \eqref{eq:conjCC} is equivalent with
\begin{align}\label{eq:conjCC'} \tag{CC'}
F(k-1,2,2)=4\left\lceil\frac{k}{4}\right\rceil.
\end{align}
If $k=4\lambda$, then \eqref{eq:conjHada} and \eqref{eq:conjCC'} are clearly equivalent. It remains to show that \eqref{eq:conjHada} implies \eqref{eq:conjCC'} for any integer $k=4\lambda+\varepsilon$ with $1\leq \varepsilon \leq 3$. 
Rao's Bound gives 
\[4\lambda <k \leq F(k-1,2,2),\]
which implies 
\begin{equation} \label{eq:4lam4}
4\lambda+4\leq F(k-1,2,2) \leq F(4\lambda+3,2,2),
\end{equation}
since $4$ divides $F(k-1,2,2)$, and $F(k,s,t)$ is non-decreasing in $k$. By \eqref{eq:conjHada} and \eqref{eq:4lam4},
\[F(k-1,2,2)=4\lambda+4=4\left\lceil\frac{k}{4}\right\rceil,\]
and the Carlet-Chen conjecture follows. 
\end{proof}

We finish this section by a partial answer to Problem \ref{prob:new} for orthogonal arrays of strength $4$.

\begin{proposition} \label{pr:apply4}
Let $k,m$ be integers, $m\geq 4$ even, with 
\[2^{m-1/2} \leq k \leq 2^m-1.\]
Then $F^*(k,2,4)=F(k,2,4)$.  
\end{proposition}
\begin{proof}
For any even integer $m\geq 4$, Kerdock \cite{KERDOCK1972182} constructed a binary, non-linear code of length $2^{m}$, cardinality $4^m$, minimum distance $2^{m-1}-2^{(m-2)/2}$ and dual distance 6. This code can be interpreted as a simple $\mathit{OA}(4^m,2^m,2,5)$, since we know that an unrestricted code has dual distance $d^\perp$ if and only if its indicator is a correlation immune function of order $d^\perp-1$ (and not of order $d^\perp$), that is, if and only if the array obtained by writing all codewords as rows is a simple OA of strength $d^\perp-1$. In the usual way, we take the rows that start with a $0$, and delete the starting $0$ to obtain a simple $\mathit{OA}(2^{2m-1},2^m-1,2,4)$. This shows
\[F(2^m-1,2,4)\leq F^*(2^m-1,2,4)\leq 2^{2m-1} \qquad \text{ for $m\geq 4$ even.}\]
Assume $2^{m-1/2} \leq k \leq 2^m-1$. Then
\begin{align*}
F(k,2,4) &\leq F(2^m-1,2,4) \\
&\leq 2^{2m-1} \\
&<2^{2m-1}+2^{m-1/2}+2\\
&\leq k^2+k+2=2 M(k,2,4).
\end{align*}
Corollary \ref{cor:1} implies  $F^*(k,2,4)=F(k,2,4)$. 
\end{proof}

We can interpret the above result in such a way that the set of integers $k$ confirming the Carlet-Guilley problem has a positive density. For any integer $t$, we define the set $\mathcal{G}(t)$ of integers $k$ such that $F^*(k,2,t)=F(k,2,t)$. Let $4\leq \mu$ be an even integer. For $4\leq m \leq \mu$ even, the set $\mathcal{G}(4)_{< 2^\mu}$ contains disjoint intervals of length
\[2^m-1-2^{m-1/2}=2^m\left(1-\frac{1}{\sqrt{2}}\right)-1.\]
Summing this up, we obtain
\begin{align*}
|\mathcal{G}(4)_{< 2^\mu}| & \geq \sum_{\text{$m\geq 4$ even}}^{\mu} 2^m\left(1-\frac{1}{\sqrt{2}}\right) -1 \\
&=\sum_{\ell=0}^{\mu/2-2} 2^{2\ell+4}\left(1-\frac{1}{\sqrt{2}}\right)-1\\
&=\frac{\mu}{2}-1 + 2^4\left(1-\frac{1}{\sqrt{2}}\right)\frac{4^{\mu/2-1}-1}{3}.
\end{align*}
Hence,
\[\lim_{\mu\to\infty}\frac{|\mathcal{G}(4)_{< 2^\mu}|}{2^\mu} \geq %
\frac{4-2\sqrt{2}}{3} \approx 0.39. \]

\begin{remark}It is not known (but not excluded either) if the Kerdock code is optimal as an unrestricted code of dual distance 6, that is, if $F^*(2^m,2,5)=4^m$ and $F^*(2^m-1,2,4)=2^{2m-1}$, for $m\geq 4$ even. It is more or less conjectured, but not yet proved explicitly, that the Preparata code of length $2^m$, with $m\geq 4$ even, is optimal as a code with size $2^{2^m-2m}$ and dual distance $2^{m-1}-2^{m/2-1}$, that is, $F^*(2^m,2,2^{m-1}-2^{m/2-1}-1)= 2^{2^m-2m}$.\end{remark}
\section{Applications and further constructions} \label{sec:applications}

\begin{proposition} \label{pr:table}
The missing entries of Table \ref{tab:minrows} are the following:
\begin{align}
\begin{split}
F^*(k,2,4)&=128 \qquad \text{ for $11\leq k \leq 15$,} \\ 
F^*(k,2,5)&=256 \qquad \text{ for $11\leq k \leq 16$,}
\end{split} \tag{A} \\
\begin{split}
F^*(12,2,6)&=768, \\
F^*(13,2,7)&=1\,536,
\end{split} \tag{B} \\
F^*(13,2,6)&=1\,024. \tag{C}
\end{align}
For all these parameters $k,t$, we have $F^*(k,2,t)=F(k,2,t)$.
\end{proposition}
\begin{proof}
(A) For $u=4$ and $k\leq 15$, shortening the Kerdock code gives
\[F(k,2,4)\leq 128, \quad\text{ and } \quad F(k+1,2,5)\leq 256.\]
If $11\leq k \leq 15$, then Corollary \ref{cor:1} implies
\begin{equation} \label{eq:FkeqFk}
F^*(k,2,4)=F(k,2,4).
\end{equation}
Assume $F(10,2,4)<128$ and let $A$ denote an $\mathit{OA}(n,10,2,4)$ with $n<128$. Then $n\leq 112$ and $A$ is simple by Theorem \ref{thm:main}. Hence, $F^*(10,2,4)\leq 112$, which contradicts to the entry
\begin{equation} \label{eq:F*10}
F^*(10,2,4)=128.
\end{equation}
of Table \ref{tab:minrows}. Hence, \eqref{eq:FkeqFk} holds for $k=10$, as well. As $F(k,s,t)$ is non-decreasing in $k$, we obtain (A).

(B) For $k=12, t=6$, Delsarte's LP Bound has value $768$. We modified the ILP method of Bulutoglu and Margot \cite{margot} to construct an array $B=\mathit{OA}(768,12,2,6)$ that has an automorphism 
\[(1,2,3,4,5)(6,7,8,9,10)\]
of order $5$. This gives rise to an array $B'=\mathit{OA}(1\,536,13,2,7)$ with weight polynomial
\[(x + 1)^{5} \cdot (x^{8} - 5 x^{7} + 28 x^{6} - 35 x^{5} + 70 x^{4} - 35 x^{3} + 28 x^{2} - 5 x + 1).\]
As shown in \cite{Krotov2020}, $B'$ is unique and it can be constructed from an equitable partition of the 13-cube. 

(C) For $k=13, t=6$, Delsarte's LP Bound has value $1\,024$. The generator matrix
\[
G=\left[\begin{array}{rrrrrrrrrrrrr}
1 & 1 & 1 & 1 & 1 & 1 & 1 & 1 & 0 & 0 & 0 & 0 & 0 \\
1 & 1 & 1 & 1 & 0 & 0 & 0 & 0 & 1 & 1 & 1 & 1 & 0 \\
1 & 1 & 0 & 0 & 1 & 1 & 0 & 0 & 1 & 1 & 0 & 0 & 1
\end{array}\right]
\]
defines a binary linear $[13,3,7]$-code $C$. The dual of $C$ is a linear $\mathit{OA}(1\,024,13,2,6)$. Notice that this construction is given in a more general context in \cite{MR4001794}. 
\end{proof}
\begin{remark}
\begin{enumerate}
\item $F(10,2,4)\geq 128$ can be deduced from \cite[Table 1]{margot}, from \cite[Table III]{eendebak}, from \cite[Appendix A]{MR4001794}, or from \cite[Theorems 18 and 20]{Boyvalenkov2017}.
\item The values in (A) are given in \cite{MR4286929}, with a more computer-based proof.
\item The true value of $F(11,2,4)$ has been asked in the Fifth International Students’ Olympiad in Cryptography NSUCRYPTO’2018 \cite[Problem ``Orthogonal arrays'']{nsu18paper}. 
\item The true value of $F^*(12,2,6)$ has been asked in the Fourth International Students’ Olympiad in Cryptography NSUCRYPTO’2017 \cite[Problem ``Masking'']{nsu17paper}. 
\item It is quite surprising that \eqref{eq:F*10} has no computer-free proof. 
\end{enumerate}
\end{remark}

\noindent \textbf{Data Availability Statement.} 
Data sharing not applicable to this article as no datasets were generated or analysed during the current study

\medspace

\noindent {\bf \large Acknowledgements}. We thank Denis Krotov, Patrick Sol\'e and Victor Zinoviev for useful information on the optimality of the Kerdock and Preparata codes. We are grateful to the anonymous referees for their insightful comments and advices which led to substantial improvements.

\printbibliography

\end{document}